\documentclass[10pt, reqno]{amsart}
\usepackage{graphicx, amssymb, amsmath, amsthm}
\numberwithin{equation}{section}
\usepackage{tikz, caption, subcaption}
\usepackage[font=small,labelfont=bf]{caption}
\usepackage{color}
\usepackage{bbm}

\usepackage{hyperref}  
\hypersetup{colorlinks=true, linkcolor=blue, anchorcolor=blue, 
citecolor=red, filecolor=blue, menucolor=blue,
urlcolor=blue}
\hypersetup{colorlinks=true, linkcolor=blue, anchorcolor=blue, 
citecolor=red, filecolor=blue, menucolor=blue,
urlcolor=blue}

\let\Im=\undefined\DeclareMathOperator*{\Im}{Im}

\newcommand{\R}{\mathbb{R}}
\newcommand{\C}{\mathbb{C}}

\newcommand{\Z}{\mathbb{Z}}

\newcommand\CC{\mathcal{C}}

\newtheorem{theorem}{Theorem}[section]

\newtheorem{lemma}[theorem]{Lemma}

\theoremstyle{definition}

\newtheorem{remark}[theorem]{Remark}

\makeatletter
\newcommand{\Extend}[5]{\ext@arrow0099{\arrowfill@#1#2#3}{#4}{#5}}
\makeatother

\begin{document}
\title[Resolvent and spectral measure]{Resolvent and spectral measure for Schr\"odinger operators on Flat Euclidean Cones }

\author{Junyong Zhang}
\address{Department of Mathematics, Beijing Institute of Technology, Beijing 100081; Department of Mathematics, Cardiff University, UK}
\email{zhang\_junyong@bit.edu.cn; ZhangJ107@cardiff.ac.uk}

\begin{abstract}
We construct the Schwartz kernel of resolvent and spectral measure for Schr\"odinger operators on the flat Euclidean cone $(X,g)$, where $X=C(\mathbb{S}_\sigma^1)=(0,\infty)\times \mathbb{S}_\sigma^1$ is a product cone over the circle, $\mathbb{S}_\sigma^1=\R/2\pi \sigma\Z$, with radius $\sigma>0$ and
the metric $g=dr^2+r^2 d\theta^2$. As products, we prove the dispersive estimates for the Schr\"odinger  and half-wave propagators in this setting.

\end{abstract}

 \maketitle

 \tableofcontents 

\section{Introduction and main results}

In this paper, we focus on the Laplacian on 2D flat Euclidean cones and construct the Schwartz kernels of resolvent and spectral measure for the Laplacian
operator in this setting. As applications, we prove the dispersive estimates for the Schr\"odinger  and half-wave propagators in this setting,
which verifies \cite[(1.7), Conjecture 1.1]{BFM} for wave and provides a simple proof of the results in \cite{Ford} for Schr\"odinger.\vspace{0.2cm}

\subsection{The motivation and the setting} 
The Schwartz kernels of the resolvent and spectral measure associated with the Schr\"odinger operators in conical singular spaces
have been systematically studied in Hassell-Vasy \cite{HV1,HV2} and Guillarmou-Hassell-Sikora \cite{GHS1, GHS2}. It has a number of applications in the
core harmonic problems (in particular the problems in $L^p$-frame) in conical singular space, for example,
the study the resolvent estimate in 
Guillarmou-Hassell \cite{GH} and 
the Strichartz 
estimates in Hassell-Zhang \cite{HZ}. Due to the generality and complexity of the geometry, the argument heavily depends on the powerful microlocal strategy developed by Melrose \cite{Melrose1,Melrose2,Melrose3} and is not standard at least for non-micrololcal readers.\vspace{0.2cm}

The purpose of this paper is to explicitly construct the kernels of the resolvent and spectral measure in a simpler conical singular space, the flat metric cone.
The argument inspired by Cheeger-Taylor \cite{CT1,CT2} is self-contained and less depends on the microlocal techniques. 
In \cite{CT1,CT2},
the authors used \emph{Lipschitz-Hankel integral formula} (e.g. \cite[Proposition 8.8]{Taylor}) to construct 
 the kernel of $\sin(t\sqrt{\Delta_g})/\sqrt{\Delta_g}$ which was used in \cite{BFM} to prove Strichartz estimates for wave equation on flat cone.
 However, the argument can not be applied to the half-wave propagator $e^{it\sqrt{\Delta_g}}$, hence they raised an open problem about dispersive estimate in \cite[(1.7), Conjecture 1.1]{BFM}. 
 One motivation of this paper is to answer this problem.
In future work, this would provide tools to study the $L^p$-frame harmonic analysis problems in this flat Euclidean cone setting or with conical singular potential (e.g. Aharonov-Bohm magnetic potential).\vspace{0.2cm}

More precisely, our setting, which is the same as studied in \cite{BFM, CT1,CT2, Ford}, is a flat Euclidean cone $(X,g)$, where $X=C(\mathbb{S}_\sigma^1)=(0,\infty)\times \mathbb{S}_\sigma^1$ is a product cone over the circle, $\mathbb{S}_\sigma^1=\R/2\pi \sigma\Z$, with radius $\sigma>0$ and
the metric $g=dr^2+r^2 d\theta^2$. This space is a generalization of the Euclidean space $\R^2$ and is a special metric cone as studied in \cite{HL, ZZ1,ZZ2}. 
We stress that this space is same to the Euclidean cone of cone angle $\alpha$, $C_\alpha=[0,\infty)_r\times (\R/\alpha \Z)_\theta$ in which the diffractive phenomena of the wave propagator was studied by \cite{FHH}.\vspace{0.2cm}

\subsection{The main results}
Our results are associated with the positive Laplace-Beltrami operator  $\Delta_g$ which is  the Friedrichs extension of the positive Laplace-Beltrami operator from the domain
$\CC_c^\infty(X)$ that consist of the compactly supported smooth functions on the
interior of the flat cone.   \vspace{0.2cm}

Our first result is about the resolvent kernel.
\begin{theorem}[Resolvent kernel]\label{thm:res}
Let $x=(r_1,\theta_1)$ and $y=(r_2,\theta_2)$ in $X=C(\mathbb{S}_\sigma^1)$. Define
\begin{equation}\label{d-j}
d_j(r_1,r_2,\theta_1,\theta_2)=\sqrt{r_1^2+r_2^2-2\cos(\theta_1-\theta_2+2j\sigma\pi)r_1r_2}
\end{equation}
and
\begin{equation} \label{d-s}
d_s(r_1,r_2,\theta_1,\theta_2)=\sqrt{r_1^2+r_2^2+2 \cosh s\, r_1r_2},\quad s\in [0,+\infty).
\end{equation}
Then the Schwartz kernel of the resolvent $$(\Delta_g-(\lambda^2\pm i0))^{-1}:=\lim_{\epsilon\searrow 0}(\Delta_g-(\lambda^2\pm i\epsilon))^{-1}$$ can be written as
the sum of the  geometry term
\begin{equation}\label{r-g}
\begin{split}
\frac{\pm i}{(4\pi)^2}
\sum_{\{j\in\Z: 0\leq |\theta_1-\theta_2+2j\sigma\pi|\leq \pi\}} H_0^{\pm}\big(\lambda d_j(r_1,r_2,\theta_1,\theta_2)\big) 
\end{split}
\end{equation}
and the diffractive term
\begin{equation}\label{r-d}
\begin{split}
-\frac{\pm i}{8\pi^3 \sigma}
 \int_0^\infty H_0^{\pm}\big(\lambda d_s(r_1,r_2,\theta_1,\theta_2)\big)\,A_{\sigma}(s,\theta_1,\theta_2) ds,
\end{split}
\end{equation}
where
\begin{equation}\label{Im-A}
A_{\sigma}(s,\theta_1,\theta_2)=\Im \Big(\frac{e^{i\frac1\sigma(\pi-(\theta_1-\theta_2))}}{e^{\frac s\sigma}-e^{i\frac1\sigma(\pi-(\theta_1-\theta_2))}}-\frac{e^{-i\frac1\sigma(\pi+(\theta_1-\theta_2))}}{e^{\frac s\sigma}-e^{-i\frac1\sigma(\pi+(\theta_1-\theta_2))}}\Big)
\end{equation}
and
 $H_0^\pm$ are the Hankel functions of order zero with $H_0^-=\overline{H_0^+}$ and, for $y>0$
\begin{equation}
H_0^+(y)= C\times \begin{cases} y^{-\frac12} e^{i(y+\frac\pi 4)}\Big(1+O(y^{-1})\Big),\qquad y\to+\infty\\
\log(\frac 2y) \Big(1+O(|\log y|)\Big),\qquad \quad y\to 0.
\end{cases}
\end{equation}
\end{theorem} 

The second result is the following about the spectral measure.
\begin{theorem} [Spectral measure kernel]\label{thm:spect}
Let $x=(r_1,\theta_1)$ and $y=(r_2,\theta_2)$ in $X=C(\mathbb{S}_\sigma^1)$ and let $d_j(r_1,r_2,\theta_1,\theta_2)$
and $d_s(r_1,r_2,\theta_1,\theta_2)$ be in \eqref{d-j} and \eqref{d-s} respectively.

Then the Schwartz kernel of the spectral measure  $dE_{\sqrt{\Delta_g}}(\lambda; x,y)$ can be written as
the sum of the following geometry term and diffractive term:
 \begin{equation}
 \begin{split}
&\frac{\lambda}{4\pi^2} \sum_{\pm}\Big(
\sum_{\{j\in\Z: 0\leq |\theta_1-\theta_2+2j\sigma\pi|\leq \pi\}}  a_\pm(\lambda d_j(r_1,r_2,\theta_1,\theta_2))e^{\pm i\lambda d_j(r_1,r_2,\theta_1,\theta_2)}
 \\&
-\frac{2}{\pi \sigma}\int_0^\infty a_\pm(\lambda d_s(r_1,r_2,\theta_1,\theta_2))e^{\pm i\lambda d_s(r_1,r_2,\theta_1,\theta_2)}
A_{\sigma}(s,\theta_1,\theta_2) ds\Big).
\end{split}
\end{equation}
where $A_{\sigma}(s,\theta_1,\theta_2)$ is given by \eqref{Im-A} and $a_\pm\in C^\infty([0,+\infty))$ satisfies
\begin{equation}\label{bean}
\begin{split}
| \partial_r^k a_\pm(r)|\leq C_k(1+r)^{-\frac{1}2-k},\quad k\geq 0.
\end{split}
\end{equation}
\end{theorem}

\section{Construction of resolvent and spectral kernels}

In this section, we first construct the representation of Schr\"odinger propagator, and then prove Theorem \ref{thm:res} and Theorem \ref{thm:spect}.

\subsection{Schr\"odinger propagator}
In this subsection, we follow the idea of Cheeger-Taylor \cite{CT1,CT2} to construct the propagator of Schr\"odinger equation.

\begin{theorem}[Schr\"odinger kernel]\label{thm:Sch-pro} Let $\Delta_g$ be the Laplacian operator on $X$ and let $x=(r_1,\theta_1)\in X$ and $y=(r_2,\theta_2)\in X$. Then the kernel of Schr\"odinger propagator 
\begin{equation}\label{S-kernel} 
\begin{split}
e^{-it\Delta_g}(x,y)&=
\frac1{4\pi}\frac{e^{-\frac{r_1^2+r_2^2}{4it}} }{it}
\sum_{\{j\in\Z: 0\leq |\theta_1-\theta_2+2j\sigma\pi|\leq \pi\}} e^{\frac{r_1r_2}{2it}\cos(\theta_1-\theta_2+2j\sigma\pi)}\\&
-\frac{1}{2\pi^2\sigma}\frac{e^{-\frac{r_1^2+r_2^2}{4it}} }{it}
 \int_0^\infty e^{-\frac{r_1r_2}{2it}\cosh s} A_{\sigma}(s,\theta_1,\theta_2) ds.
\end{split}
\end{equation}
where $A_{\sigma}(s,\theta_1,\theta_2)$ is given in \eqref{Im-A}.

\end{theorem}

\begin{remark} If $\sigma=1$, then $A_{\sigma}(s,\theta_1,\theta_2)$ vanishes. The first term becomes
$$(4\pi it)^{-1}e^{-\frac{|x-y|^2}{4it}}, $$
which consists with the kernel of Schr\"odinger propagator in Euclidean space.
\end{remark}

\begin{proof} We first recall Cheeger's functional calculus for 2D flat cone. For more details, we refer to \cite{CT1, Taylor}. 
We write the positive Laplacian on $X$ 
\begin{equation*}
\Delta_g=-\partial_r^2-\frac1r\partial_r+\frac1{r^2}\Delta_{\mathbb{S}_\sigma^1}
\end{equation*}
where $\Delta_{\mathbb{S}_\sigma^1}=-\partial_\theta^2$ is the Laplacian operator on $\mathbb{S}_\sigma^1$. 
Let $\nu_k=|k|/\sigma$ and 
\begin{equation}
\varphi_k(\theta)=\frac1{\sqrt{2\pi\sigma}}e^{-\frac{ik\theta}{\sigma}},\qquad k\in\Z,
\end{equation}
then $\nu_k$ and $\varphi_k$ are eigenvalues and eigenfunctions of operator $\Delta_{\mathbb{S}_\sigma^1}$ such that
\begin{equation}
-\partial_\theta^2\varphi_k(\theta)=\nu_k^2\varphi_k(\theta).
\end{equation}
By Cheeger’s separation of variables functional calculus (e.g. \cite[(8.47)]{Taylor}),  we obtain the kernel of the operator  $e^{-it\Delta_g}$
\begin{equation}\label{equ:kernsina}
K(t,x,y)=K(t,r_1,\theta_1,r_2,\theta_2)=\sum_{k\in\Z}\varphi_{k}(\theta_1)\overline{\varphi_{k}(\theta_2)}K_{\nu_k}(t,r_1,r_2)=K_{\nu}(t,r_1,r_2)
\end{equation}
where $K_{\nu}(t,r_1,r_2)$ is given by
\begin{equation}\label{equ:knukdef}
\begin{split}
  K_{\nu}(t,r_1,r_2)&=\int_0^\infty e^{-it\rho^2}J_{\nu}(r_1\rho)J_{\nu}(r_2\rho) \,\rho d\rho
  \\&=\lim_{\epsilon\searrow 0} \int_0^\infty e^{-(\epsilon+it)\rho^2}J_{\nu}(r_1\rho)J_{\nu}(r_2\rho) \,\rho d\rho.
  \end{split}
\end{equation}
and $\nu$ is an operator defined on $\mathbb{S}_\sigma^1$ by
\begin{equation}\label{ope-nu}
\nu=\sqrt{\Delta_{\mathbb{S}_\sigma^1}},\quad \nu \varphi_k(\theta)=\nu_k \varphi_k(\theta).
\end{equation}
By using the Weber's second exponential integral \cite[Section 13.31 (1)]{Watson}, we show, for $\epsilon>0$ 
\begin{equation}
\begin{split}
 \int_0^\infty  e^{-(\epsilon+it)\rho^2}  J_{\nu}(r_1\rho) J_{\nu}(r_2\rho)\rho d\rho =\frac{e^{-\frac{r_1^2+r_2^2}{4(\epsilon+it)}}}{2(\epsilon+it)} I_\nu\big(\frac{r_1r_2}{2(\epsilon+it)}\big),
\end{split}
\end{equation}
where $I_\nu(x)$ is the modified Bessel function of the first kind
\begin{equation*}
\begin{split}
I_\nu(x)=\sum_{j=0}^\infty \frac{1}{j!\Gamma(\nu+j+1)}\big(x/2\big)^{\nu+2j}.
\end{split}
\end{equation*}
Define $$z=\frac{r_1r_2}{2(\epsilon+it)}, \quad \epsilon>0, $$
and recall the integral representation with 
\begin{equation*}
I_\nu(z)=\frac1{\pi}\int_0^\pi e^{z\cos(s)} \cos(\nu s) ds-\frac{\sin(\nu\pi)}{\pi}\int_0^\infty e^{-z\cosh s} e^{-s\nu} ds,
\end{equation*}
then
\begin{equation}\label{equ:knukdef}
\begin{split}
 & K_{\nu}(t,r_1,r_2)=\lim_{\epsilon\searrow 0}\frac{e^{-\frac{r_1^2+r_2^2}{4(\epsilon+it)}}}{2(\epsilon+it)} I_\nu\big(\frac{r_1r_2}{2(\epsilon+it)}\big)\\
  &=\lim_{\epsilon\searrow 0}\frac{e^{-\frac{r_1^2+r_2^2}{4(\epsilon+it)}}}{2(\epsilon+it)}
  \Big(\frac1{\pi}\int_0^\pi e^{z\cos(s)} \cos(\nu s) ds-\frac{\sin(\nu\pi)}{\pi}\int_0^\infty e^{-z\cosh s} e^{-s\nu} ds\Big).
  \end{split}
\end{equation}
Recall $\nu$ is an operator in \eqref{ope-nu}, we identity the operator with its kernel
\begin{equation}
\begin{split}
\cos(\nu s)=\cos( s\sqrt{\Delta_{\mathbb{S}^1_\sigma}})&= \sum_{k\in\Z}\varphi_{k}(\theta_1)\overline{\varphi_{k}(\theta_2)}\cos(\nu_k s)
\\&=
\sum_{k\in\Z}\frac{1}{2\pi\sigma}e^{-i\frac{k}{\sigma}(\theta_1-\theta_2)}\frac{e^{i\frac{ks}{\sigma}}+e^{-i\frac{ks}{\sigma}}}2.
\end{split}
\end{equation} 

Note that formula about the relation between the Dirac comb distribution and its Fourier series
\begin{equation}
\sum_{j\in\Z} \delta(x-Tj)=\sum_{k\in\Z} \frac1T e^{i 2\pi \frac{k}{T}x}, \quad T=2\pi \sigma,
\end{equation}
 we  obtain
 \begin{equation}
 \begin{split}
 \cos(\nu s)=
 \frac12\sum_{j\in\Z}\big[\delta(\theta_1-\theta_2+s+2j\pi \sigma)
 +\delta(\theta_1-\theta_2-s +2j\pi\sigma)\big].
 \end{split}
 \end{equation}
Thus we write the first term in the bracket 
\begin{align*}
&\frac1{\pi}\int_0^\pi e^{z\cos(s)} \cos(\nu s) ds\\ \nonumber
  =&\frac1{2\pi}\sum_{j\in\Z}\int_0^{\pi} e^{z\cos(s)}\big[\delta(\theta_1-\theta_2+s+2j\sigma\pi)
 +\delta(\theta_1-\theta_2-s+2j\sigma\pi)\big]\;ds\\\nonumber
 =&\frac1{2\pi}\sum_{\{j\in\Z: 0\leq |\theta_1-\theta_2+2j\sigma\pi|\leq \pi\}} e^{z\cos(\theta_1-\theta_2+2j\sigma\pi)}.
\end{align*}
Therefore the contribution of the first term is
\begin{equation}\label{I-g}
\frac{e^{-\frac{r_1^2+r_2^2}{4(\epsilon+it)}} }{4\pi(\epsilon+it)}
\sum_{\{j\in\Z: 0\leq |\theta_1-\theta_2+2j\sigma\pi|\leq \pi\}} e^{\frac{r_1r_2}{2(\epsilon+it)}\cos(\theta_1-\theta_2+2j\sigma\pi)}.
\end{equation}
We next consider the second term in the bracket 
 \begin{align}
\nonumber&\frac1{2\pi \sigma}\sum_{k\in\Z} e^{-i\frac{k}{\sigma}(\theta_1-\theta_2)}\frac{\sin(\nu_k\pi)}{\pi}\int_0^\infty e^{-z\cosh s} e^{-s\nu_k} ds.
\end{align}
Recall $\nu_k=|k|/\sigma$, therefore we furthermore have
\begin{equation}
\begin{split}
&\sum_{k\in\Z}\sin(\pi\frac{|k|}\sigma)e^{-\frac{|k|s}{\sigma}}e^{-i\frac k\sigma(\theta_1-\theta_2)}\\
=&\sum_{k\geq 1} \frac{e^{i\frac{k}{\sigma}\pi}-e^{-i\frac{k}{\sigma}\pi}}{2i} e^{-\frac{ks}{\sigma}}e^{-i\frac k{\sigma}(\theta_1-\theta_2)}+\sum_{k\leq -1} \frac{e^{-i\frac{k}{\sigma}\pi}-e^{i\frac{k}{\sigma}\pi}}2 e^{\frac{ks}{\sigma}}e^{-i\frac k{\sigma}(\theta_1-\theta_2)}\\
=&\sum_{k\geq 1} \frac{e^{i\frac{k}{\sigma}\pi}-e^{-i\frac{k}{\sigma}\pi}}{2i} e^{-\frac{ks}{\sigma}}\big(e^{-i\frac k{\sigma}(\theta_1-\theta_2)}+e^{i\frac k{\sigma}(\theta_1-\theta_2)}\big)\\
=& \sum_{k\geq 1}\Im \big(e^{i\frac{k}\sigma(\pi-(\theta_1-\theta_2)+is)}-e^{i\frac{k}\sigma(-\pi-(\theta_1-\theta_2)+is)}\big).
\end{split}
\end{equation}
Note that
\begin{equation}
\sum_{k=1}^\infty e^{ikz}=\frac{e^{iz}}{1-e^{iz}},\qquad \mathrm{Im} z>0,
\end{equation}
we finally obtain
\begin{align}
&\sum_{k\in\Z}\sin(\pi\frac{|k|}\sigma)e^{-\frac{|k|s}{\sigma}}e^{-i\frac k\sigma(\theta_1-\theta_2)}\\ \nonumber
=& \Im \big(\frac{e^{i\frac1\sigma(\pi-(\theta_1-\theta_2))}}{e^{\frac s\sigma}-e^{i\frac1\sigma(\pi-(\theta_1-\theta_2))}}-\frac{e^{-i\frac1\sigma(\pi+(\theta_1-\theta_2))}}{e^{\frac s\sigma}-e^{-i\frac1\sigma(\pi+(\theta_1-\theta_2))}}\big)\\
=&A_{\sigma}(s,\theta_1,\theta_2).
\end{align}
Therefore we obtain the contribution of the second term
\begin{equation}\label{I-d}
\begin{split}
&-\frac{1}{2\pi^2\sigma}\frac{e^{-\frac{r_1^2+r_2^2}{4(\epsilon+it)}} }{(\epsilon+it)}
 \int_0^\infty e^{-\frac{r_1r_2}{2(\epsilon+it)}\cosh s} A_{\sigma}(s,\theta_1,\theta_2) ds.
\end{split}
\end{equation}
Collecting \eqref{I-d} and \eqref{I-g} and letting $\epsilon\searrow 0$, we obtain \eqref{S-kernel}, the fundamental solution of Schr\"odinger equation 
\begin{equation*}
\frac1{4\pi}\frac{e^{-\frac{r_1^2+r_2^2}{4it}} }{it}
\sum_{\{j\in\Z: 0\leq |\theta_1-\theta_2+2j\sigma\pi|\leq \pi\}} e^{\frac{r_1r_2}{2it}\cos(\theta_1-\theta_2+2j\sigma\pi)}.
\end{equation*}

\begin{equation*}
\begin{split}
&-\frac{1}{2\pi^2\sigma}\frac{e^{-\frac{r_1^2+r_2^2}{4it}} }{it}
 \int_0^\infty e^{-\frac{r_1r_2}{2it}\cosh s} A_{\sigma}(s,\theta_1,\theta_2) ds.
\end{split}
\end{equation*}

\end{proof}

\subsection{The resolvent kernel}
In this subsection, we prove Theorem \ref{thm:res} from the Schr\"odinger kernel in Theorem \ref{thm:Sch-pro}.
Note that
 \begin{equation}\label{out-inc}
\begin{split}
(\Delta_g-(\lambda^2-i0))^{-1}=\overline{(\Delta_g-(\lambda^2+i0))^{-1}},
\end{split}
\end{equation}
we only construct the kernel of 
 \begin{equation}
\begin{split}
(\Delta_g-(\lambda^2+i0))^{-1}.
\end{split}
\end{equation}
We first note that
when $z\in \{z\in\C: \Im(z)>0\}$, then
$$(s-z)^{-1}=\frac1{i}\int_0^\infty e^{-ist} e^{iz t}dt,\quad \forall s\in\R,$$
thus we obtain, for $z=\lambda^2+i\epsilon$ with $\epsilon>0$ 
\begin{equation}\label{res+}
\begin{split}
(\Delta_g-(\lambda^2+i0))^{-1}&=\frac1{i}\lim_{\epsilon\to 0^+}\int_0^\infty e^{-it\Delta_g} e^{it(\lambda^2+i\epsilon)}dt.
\end{split}
\end{equation}
To prove Theorem \ref{thm:res}, we need two lemmas.  

\begin{lemma} Let  $z=\lambda^2+i\epsilon$ with $\epsilon>0$. Then
\begin{equation}\label{equ:m}
\begin{split}
\int_0^\infty \frac{e^{-\frac{r_1^2+r_2^2}{4it}} }{it} e^{\frac{r_1r_2}{2it} \cos(\theta_1-\theta_2+2j\sigma\pi)} e^{iz t}dt=\frac{i}\pi\int_{\R^2} \frac{e^{-i{\bf m}\cdot {\xi}}}{|\xi|^2-z} \, d{ \xi},
\end{split}
\end{equation}
and
\begin{equation}\label{equ:n}
\begin{split}
\int_0^\infty \frac{e^{-\frac{r_1^2+r_2^2}{4it}} }{it} e^{-\frac{r_1r_2}{2it}\cosh s} e^{iz t}dt=\frac i \pi\int_{\R^2}\frac{e^{-i{\bf n}\cdot {\xi}}}{|\xi|^2-z} \, d{ \xi},
\end{split}
\end{equation}
where $\xi=(\xi_1,\xi_2)\in\R^2$ and ${\bf m}, {\bf n}\in\R^2$ such that
\begin{equation}\label{bf-m}
{\bf m}=(r_1-r_2, \sqrt{2(1-\cos(\theta_1-\theta_2+2j\sigma\pi))r_1r_2}),
\end{equation}
and
\begin{equation}\label{bf-n}
{\bf n}=(n_1,n_2)=\big(r_1+r_2, \sqrt{2(\cosh s-1)r_1r_2}\big).
\end{equation}

\end{lemma}

\begin{proof} We first prove \eqref{equ:m}. Let
\begin{equation}\label{s-r}
s_1=r_1-r_2, \quad s_2=\sqrt{2r_1r_2}.
\end{equation}
 We write
\begin{equation}
\begin{split}
&\frac{e^{-\frac{r_1^2+r_2^2}{4it}} }{it} e^{\frac{r_1r_2}{2it}\cos(\theta_1-\theta_2+2j\sigma\pi)} =
\frac{e^{-\frac{(r_1-r_2)^2}{4it}} }{\sqrt{it}} \frac{e^{-\frac{r_1r_2}{2it}(1-\cos(\theta_1-\theta_2+2j\sigma\pi))}}{\sqrt{it}}\\
&=\frac{e^{-\frac{s_1^2}{4it}} }{\sqrt{it}} \frac{e^{-\frac{s_2^2}{4it}(1-\cos(\theta_1-\theta_2+2j\sigma\pi))}}{\sqrt{it}}.
\end{split}
\end{equation} 
By using the formula
\begin{equation}\label{F-Gaussian}
\int_{-\infty}^\infty e^{-it\eta^2} e^{-i \eta r} d\eta=\sqrt{\frac{\pi}{it}}e^{-\frac{ r^2}{4it}},
\end{equation}
we have 
\begin{equation}
\begin{split}
&\frac{e^{-\frac{r_1^2+r_2^2}{4it}} }{it} e^{\frac{r_1r_2}{2it}\cos(\theta_1-\theta_2+2j\sigma\pi)} \\
&=\frac1\pi\int_{-\infty}^\infty e^{-it\xi_1^2} e^{-is_1\xi_1} d\xi_1 \int_{-\infty}^\infty e^{-it\xi_2^2} e^{-i\sqrt{1-\cos(\theta_1-\theta_2+2j\sigma\pi)}s_2\xi_2} d\xi_2\\
&=\frac1\pi\int_{\R^2} e^{-it(\xi_1^2+\xi_2^2)} e^{-i{\bf m}\cdot(\xi_1,\xi_2)} d\xi_1 d\xi_2.
\end{split}
\end{equation} 
where ${\bf m}$ is given by \eqref{bf-m}. Let $\xi=(\xi_1,\xi_2)$, then 
\begin{equation}
\begin{split}
&\int_0^\infty \frac{e^{-\frac{r_1^2+r_2^2}{4it}} }{it} e^{\frac{r_1r_2}{2it} \cos(\theta_1-\theta_2+2j\sigma\pi)} e^{iz t}dt\\
&=\frac1 \pi\int_{\R^2} e^{-i{\bf m}\cdot\xi} \int_0^\infty e^{-it|\xi|^2} e^{iz t}dt\, d\xi\\
&=\frac i \pi\int_{\R^2} \frac{e^{-i{\bf m}\cdot\xi}}{|\xi|^2-z} \, d\xi
\end{split}
\end{equation}
which shows \eqref{equ:m}. Next we prove \eqref{equ:n}. We consider 
\begin{equation}
\begin{split}
\int_0^\infty \frac{e^{-\frac{r_1^2+r_2^2}{4it}} }{it} e^{-\frac{r_1r_2}{2it}\cosh s} e^{iz t}dt.
\end{split}
\end{equation}
Instead of \eqref{s-r}, using the variable changes $$s_1=r_1+r_2, \quad s_2=\sqrt{2r_1r_2},$$ we similarly write
\begin{equation}
\begin{split}
\frac{e^{-\frac{r_1^2+r_2^2}{4it}} }{it} e^{-\frac{r_1r_2}{2it}\cosh s} &=
\frac{e^{-\frac{(r_1+r_2)^2}{4it}} }{\sqrt{it}} \frac{e^{-\frac{r_1r_2}{2it}(\cosh s-1)}}{\sqrt{it}}\\
&=\frac{e^{-\frac{s_1^2}{4it}} }{\sqrt{it}} \frac{e^{-\frac{s_2^2}{4it}(\cosh s-1)}}{\sqrt{it}}.
\end{split}
\end{equation} 
By \eqref{F-Gaussian}, we obtain that 
\begin{equation}
\begin{split}
&\frac{e^{-\frac{s_1^2}{4it}} }{\sqrt{it}} \frac{e^{-\frac{s_2^2}{4it}(\cosh s-1)}}{\sqrt{it}}\\
&=\frac 1\pi\int_{-\infty}^\infty e^{-it\xi_1^2}e^{-is_1\xi_1} d\xi_1 \int_{-\infty}^\infty e^{-it\xi_2^2}e^{-i\sqrt{\cosh s-1}s_2\xi_2} d\xi_2\\
&=\frac 1\pi\int_{-\infty}^\infty \int_{-\infty}^\infty e^{-it(\xi_1^2+\xi_2^2)}e^{-i(s_1\xi_1+\sqrt{\cosh s-1}s_2\xi_2)} d\xi_1 d\xi_2.
\end{split}
\end{equation} 
Similarly as above, we show
\begin{equation}
\begin{split}
\int_0^\infty \frac{e^{-\frac{r_1^2+r_2^2}{4it}} }{it} e^{-\frac{r_1r_2}{2it}\cosh s} e^{iz t}dt
=\frac {i}\pi\int_{\R^2} \frac{e^{-i{\bf n}\cdot\xi}}{|\xi|^2-z} d\xi
\end{split}
\end{equation}
where ${\bf n}=(n_1,n_2)=(r_1+r_2, \sqrt{2(\cosh s-1)r_1r_2})$.

\end{proof}

Next, we need the following lemma about the resolvent for Laplacian in $\R^2$ which is well-known, for example, see \cite[Chapter 3.4]{CK1} or \cite[Eq. (5.16.3)]{Le}. 
\begin{lemma} Let  $z=\lambda^2+i\epsilon$ with $\epsilon>0$. Then
\begin{equation}\label{equ:res-free}
\begin{split}
\lim_{\epsilon\to0^+}\frac1\pi\int_{\R^2} \frac{e^{-ix\cdot {\xi}}}{|\xi|^2-z} \, d{ \xi}=\frac i {4\pi} H_0^+(\lambda |x|),\quad x\in\R^2\setminus\{0\},
\end{split}
\end{equation}
where $H_0^+$ is the Hankel function of order zero and let $y=\lambda|x|>0$
\begin{equation}
H_0^+(y)= C\times \begin{cases} y^{-\frac12} e^{i(y+\frac\pi 4)}\Big(1+O(y^{-1})\Big),\qquad y\to+\infty\\
\log(\frac 2y) \Big(1+O(|\log y|)\Big),\qquad \quad y\to 0.
\end{cases}
\end{equation}

\end{lemma}

Now we prove Theorem \ref{thm:res}.
From \eqref{res+} and Theorem \ref{thm:Sch-pro}, we obtain
\begin{equation*}
\begin{split}
&(\Delta_g-(\lambda^2+i0))^{-1}\\&=
\frac1{4\pi i}
\sum_{\{j\in\Z: 0\leq |\theta_1-\theta_2+2j\sigma\pi|\leq \pi\}} \lim_{\epsilon\to 0^+}\int_0^\infty e^{it(\lambda^2+i\epsilon)}\frac{e^{-\frac{r_1^2+r_2^2}{4it}} }{it} e^{\frac{r_1r_2}{2it}\cos(\theta_1-\theta_2+2j\sigma\pi)} \,dt\\&
-\frac{1}{2\pi^2\sigma i}
 \int_0^\infty \lim_{\epsilon\to 0^+}\int_0^\infty e^{it(\lambda^2+i\epsilon)} \frac{e^{-\frac{r_1^2+r_2^2}{4it}} }{it} e^{-\frac{r_1r_2}{2it}\cosh s} \,dt A_{\sigma}(s,\theta_1,\theta_2) ds.
 \end{split}
\end{equation*}
From \eqref{equ:m} and \eqref{equ:n}, it follows 
\begin{equation}\label{equ:res+}
\begin{split}
&(\Delta_g-(\lambda^2+i0))^{-1}\\&=
\frac1{4\pi}
\sum_{\{j\in\Z: 0\leq |\theta_1-\theta_2+2j\sigma\pi|\leq \pi\}} \lim_{\epsilon\to 0^+}\frac1\pi\int_{\R^2} \frac{e^{-i{\bf m}\cdot {\xi}}}{|\xi|^2-(\lambda^2+i\epsilon)} \, d{ \xi}\\&
-\frac{1}{2\pi^2\sigma }
 \int_0^\infty \lim_{\epsilon\to 0^+}\frac1\pi\int_{\R^2} \frac{e^{-i{\bf n}\cdot {\xi}}}{|\xi|^2-(\lambda^2+i\epsilon)} \, d{ \xi} \,A_{\sigma}(s,\theta_1,\theta_2) ds.
 \end{split}
\end{equation}
From \eqref{bf-m} and \eqref{bf-n}, we have that
$$|{\bf m}|=\sqrt{r_1^2+r_2^2-2\cos(\theta_1-\theta_2+2j\sigma\pi)r_1r_2}=d_j(r_1,r_2,\theta_1,\theta_2)$$
and
$$|{\bf n}|=\sqrt{r_1^2+r_2^2+2\cosh s\, r_1r_2}=d_s(r_1,r_2,\theta_1,\theta_2).$$
Therefore,
from \eqref{equ:res-free},  we prove
\begin{equation*}
\begin{split}
&(\Delta_g-(\lambda^2+i0))^{-1}\\&=
\frac{i}{(4\pi)^2}
\sum_{\{j\in\Z: 0\leq |\theta_1-\theta_2+2j\sigma\pi|\leq \pi\}} H_0^{+}\big(\lambda d_j(r_1,r_2,\theta_1,\theta_2)\big) \\&
-\frac{i}{8\pi^3\sigma}
 \int_0^\infty H_0^{+}\big(\lambda d_s(r_1,r_2,\theta_1,\theta_2)\big)\,A_{\sigma}(s,\theta_1,\theta_2) ds.
\end{split}
\end{equation*}
which implies Theorem \ref{thm:res}.

\subsection{The spectral measure}
 In this subsection, we show a representation of the spectral measure in Theorem \ref{thm:spect}.
According to Stone’s formula, the spectral measure is related to the resolvent 
 \begin{equation}\label{stone}
 dE_{\sqrt{-\Delta_g}}(\lambda)=\frac{d}{d\lambda}dE_{\sqrt{-\Delta_g}}(\lambda)\,d\lambda=\frac{\lambda}{\pi i}\big(R(\lambda+i0)-R(\lambda-i0)\big)\, d\lambda
 \end{equation}
 where the resolvent $$R(\lambda\pm i0)=\lim_{\epsilon\searrow 0}(\Delta_g-(\lambda^2\pm i\epsilon))^{-1}.$$
 
From \eqref{stone}, \eqref{out-inc} and \eqref{equ:res+}, we have
 \begin{equation*}
 \begin{split}
 &dE_{\sqrt{-\Delta_g}}(\lambda;x,y)\\
& =
\frac1{4\pi} \frac{\lambda}{\pi^2 i}
\sum_{\{j\in\Z: 0\leq |\theta_1-\theta_2+2j\sigma\pi|\leq \pi\}} \int_{\R^2} e^{-i{\bf m}\cdot {\xi}}\Big(\frac1{|\xi|^2-(\lambda^2+i0)}-\frac1{|\xi|^2-(\lambda^2-i0)}\Big) \, d{ \xi}
 \\&
-\frac{1}{2\pi^2 \sigma} \frac{\lambda}{\pi^2 i}
 \int_0^\infty \int_{\R^2} e^{-i{\bf n}\cdot {\xi}}\Big(\frac1{|\xi|^2-(\lambda^2+i0)}-\frac1{|\xi|^2-(\lambda^2-i0)}\Big) \, d{ \xi}
A_{\sigma}(s,\theta_1,\theta_2) ds.
\end{split}
\end{equation*}
On the one hand, we note that
\begin{equation}\label{id-spect}
\begin{split}
&\lim_{\epsilon\to 0^+}\frac{\lambda}{\pi i}\int_{\R^2} e^{-ix\cdot\xi}\Big(\frac{1}{|\xi|^2-(\lambda^2+i\epsilon)}-\frac{1}{|\xi|^2-(\lambda^2-i\epsilon)}\Big) d\xi\\
&=\lim_{\epsilon\to 0^+} \frac{\lambda}{\pi }\int_{\R^2} e^{-ix\cdot\xi}\Im\Big(\frac{1}{|\xi|^2-(\lambda^2+i\epsilon)}\Big)d\xi\\
&=\lim_{\epsilon\to 0^+} \frac{\lambda}{\pi }\int_{0}^\infty \frac{\epsilon}{(\rho^2-\lambda^2)^2+\epsilon^2} \int_{|\omega|=1} e^{-i\rho x\cdot\omega} d\sigma_\omega  \, \rho d\rho\\
&= \lambda \int_{|\omega|=1} e^{-i\lambda x\cdot\omega} d\sigma_\omega  \\
\end{split}
\end{equation}
where we use the fact 
 the  Poisson kernel is is an approximation to the identity which implies that, for any reasonable function $m(x)$
\begin{equation}
\begin{split}
m(x)&=\lim_{\epsilon\to 0^+}\frac1\pi \int_{\R} \Im\big(\frac{1}{x-(y+i\epsilon)}\big) m(y)dy
\\&=\lim_{\epsilon\to 0^+}\frac1\pi \int_{\R} \frac{\epsilon}{(x-y)^2+\epsilon^2} m(y)dy.
\end{split}
\end{equation}
On the other hand, for example \cite[Theorem 1.2.1]{sogge}, we also note that
\begin{equation}
\begin{split}
\int_{\mathbb{S}^{1}} e^{-i x\cdot\omega} d\sigma(\omega)=\sum_{\pm}  a_\pm(|x|) e^{\pm i|x|}
\end{split}
\end{equation}
where 
\begin{equation}
\begin{split}
| \partial_r^k a_\pm(r)|\leq C_k(1+r)^{-\frac{1}2-k},\quad k\geq 0.
\end{split}
\end{equation}
Therefore we obtain that
 \begin{equation*}
 \begin{split}
 dE_{\sqrt{-\Delta_g}}(\lambda;x,y)& =
\frac{\lambda}{4\pi^2} \sum_{\pm}\Big(
\sum_{\{j\in\Z: 0\leq |\theta_1-\theta_2+2j\sigma\pi|\leq \pi\}}  a_\pm(\lambda |{\bf{m}}|)e^{\pm i\lambda |{\bf m}|}
 \\&
-\frac{2}{ \pi \sigma}\int_0^\infty a_\pm(\lambda |{\bf{n}}|)e^{\pm i\lambda |{\bf n}|}
A_{\sigma}(s,\theta_1,\theta_2) ds\Big).
\end{split}
\end{equation*}
Note $|{\bf m}|=d_j$ and $|{\bf n}|=d_s$ in \eqref{d-j} and \eqref{d-s} again, we prove Theorem \ref{thm:spect}.

\section{Applications}

\subsection{Dispersive estimates} In this subsection, we prove the dispersive estimates for the Schr\"odinger and half-wave propagator. More precisely, we will prove

\begin{theorem}\label{thm:dis-S} Let $e^{-it\Delta_g}$ be \eqref{S-kernel} the fundamental solution of Schr\"odinger equation. Then the dispersive estimate holds
\begin{equation}\label{dis-s}
|e^{-it\Delta_g}(x,y)|\leq C_\sigma |t|^{-1}.
\end{equation}
where $C_\sigma$ is a constant independent of $x,y$.
\end{theorem}
\begin{remark}The dispersive estimate was first proved in \cite{Ford}. We provide a new proof based on Theorem \ref{thm:Sch-pro}.
\end{remark}

\begin{theorem}[Dispersive estimates for wave]\label{thm:dis-w} Let $\Delta_g$ be the positive Laplacian operator on $X$ and
suppose $\phi\in C_c^\infty([1/2, 2])$ and takes value in
$[0,1]$. Assume $f=\phi(2^{-k}\sqrt{\Delta_g})$ with $k\in \Z$, then there exists a constant $C$ independent of $t$ and $k\in \Z$ such that
\begin{equation}\label{dispersive}
\begin{split}
\|e^{it\sqrt{\Delta_g}} f\|_{L^\infty(X)}\leq C 2^{\frac32 k}(2^{-k}+|t|)^{-\frac12}\|f\|_{L^1(X)}.
\end{split}
\end{equation}
\end{theorem}

\begin{remark} The result verifies  \cite[(1.7), Conjecture 1.1]{BFM}. As a consequence, the dispersive estimate implies  the Strichartz estimates 
by following Keel-Tao \cite{KT}, or \cite{Zhang}.
\end{remark}

\begin{proof}[The proof of Theorem \ref{thm:dis-S}] We begin with Theorem \ref{thm:Sch-pro}. For given $\sigma>0$, the summation of $j$ in the first term is finite and bounded by $O(1+\frac1{\sigma})$. Thus it is easy to see
\begin{equation*}
\Big|\frac1{4\pi}\frac{e^{-\frac{r_1^2+r_2^2}{4it}} }{it}
\sum_{\{j\in\Z: 0\leq |\theta_1-\theta_2+2j\sigma\pi|\leq \pi\}} e^{\frac{r_1r_2}{2it}\cos(\theta_1-\theta_2+2j\sigma\pi)}\Big|\leq C_\sigma |t|^{-1}.
\end{equation*}
Now we estimate the second term. Let
$$\phi=\frac1\sigma(\pi-(\theta_1-\theta_2))\quad \text{or}\quad \frac1\sigma(-\pi-(\theta_1-\theta_2)),$$
then, to prove \eqref{dis-s}, it suffices to prove
\begin{equation}\label{0-infty}
\begin{split}
 \int_0^\infty  \Big|\Im \Big(\frac{e^{i\phi}}{e^{s}-e^{i\phi}}\Big)\Big| ds\leq C
\end{split}
\end{equation}
where $C$ is a constant independent of $\phi$. To see this, we first have
\begin{equation*}
\begin{split}
 \int_1^\infty \Big|\Im \Big(\frac{e^{i\phi}}{e^{s}-e^{i\phi}}\Big)\Big| ds\leq  \int_1^\infty e^{-s/2} ds\leq C.
\end{split}
\end{equation*}
On the other hand, we have 
\begin{equation}
\begin{split}
\Im \Big(\frac{e^{i\phi}}{e^{s}-e^{i\phi}}\Big)&=\Im \Big(\frac{(\cos\phi+i\sin\phi)(e^s-\cos\phi+i\sin\phi)}{(e^{s}-\cos \phi)^2+\sin^2\phi}\Big)\\
&=
\frac{e^s\, \sin\phi }{(e^{s}-\cos \phi)^2+\sin^2\phi}.
\end{split}
\end{equation}
Therefore we obtain 
\begin{equation*}
\begin{split}
 \int_0^1 \Big|\Im \Big(\frac{e^{i\phi}}{e^{s}-e^{i\phi}}\Big)\Big| ds&\leq  \int_0^1 \frac{e^s\, |\sin\phi| }{(e^{s}-\cos \phi)^2+\sin^2\phi} ds\\
 &\leq \int_0^3 \frac{ |\sin\phi| }{s^2+\sin^2\phi} ds\leq  \int_0^\infty \frac{1}{s^2+1} ds\leq C.
\end{split}
\end{equation*}
Thus we obtain \eqref{0-infty}, hence prove \eqref{dis-s}.

\end{proof}

\begin{proof}[The proof of Theorem \ref{thm:dis-w}] The crucial points are Theorem \ref{thm:spect} and stationary phase argument. 
We write 
 \begin{equation*}
e^{it\sqrt{\Delta_g}}f=\int_X \int_0^\infty e^{it\lambda} \phi(2^{-k}\lambda) dE_{\sqrt{\Delta_g}}(\lambda; x, y) f(y) dy
  \end{equation*}
  Then it suffices to show kernel estimate 
  \begin{equation}\label{est:k}
\Big|  \int_0^\infty e^{it\lambda} \phi(2^{-k}\lambda) dE_{\sqrt{\Delta_g}}(\lambda; x, y) \Big|\leq C 2^{\frac32 k}(2^{-k}+|t|)^{-\frac12}.
\end{equation}
To this end, from Theorem \ref{thm:spect}, we aim to estimate 
  \begin{equation}\label{est:k1}
\Big|  \int_0^\infty e^{it\lambda} \phi(2^{-k}\lambda)\lambda a_\pm(\lambda d_j)e^{\pm i\lambda d_j} d\lambda\Big|\leq C 2^{\frac32 k}(2^{-k}+|t|)^{-\frac12}.
\end{equation}
and
  \begin{equation}\label{est:k2}
  \begin{split}
\Big|  \int_0^\infty e^{it\lambda} \phi(2^{-k}\lambda)\lambda \int_0^\infty a_\pm(\lambda d_s)e^{\pm i\lambda d_s}
&A_{\sigma}(s,\theta_1,\theta_2) ds \,d\lambda\Big|\\&\leq C 2^{\frac32 k}(2^{-k}+|t|)^{-\frac12}.
\end{split}
\end{equation}
where $a_\pm$ satisfies \eqref{bean}, 
and $d_j, d_s$ are in \eqref{d-j} and \eqref{d-s}, and $A_\sigma$ is given by \eqref{Im-A}.
Since $a_\pm$ satisfies \eqref{bean}, let $d=d_j$ or $d_s$, hence
\begin{equation}\label{bean'}
 |\partial_\lambda^N [a_\pm(\lambda d)]|\leq C_N \lambda^{-N}(1+\lambda d)^{-\frac{1}2},\quad N\geq 0.
\end{equation}
We first prove \eqref{est:k1}. By \eqref{bean'}, we use the $N$-times integration by parts to obtain 
\begin{equation*}
\begin{split}
&\Big|  \int_0^\infty e^{it\lambda} \phi(2^{-k}\lambda)\lambda a_\pm(\lambda d_j)e^{\pm i\lambda d_j} d\lambda\Big|\\&\leq \Big|\int_0^\infty \left(\frac1{
i(t\pm d_j)}\frac\partial{\partial\lambda}\right)^{N}\big(e^{i(t\pm d_j)\lambda}\big)
\phi(2^{-k}\lambda)\lambda a_\pm(\lambda d_j) d\lambda\Big|\\& \leq
C_N|t\pm d_j|^{-N}\int_{2^{k-1}}^{2^{k+1}}\lambda^{1-N}(1+\lambda
d_j)^{-1/2}d\lambda\\&\leq
C_N2^{k(2-N)}|t\pm d_j |^{-N}(1+2^kd_j)^{-1/2}.
\end{split}
\end{equation*}
It follows that
\begin{equation}\label{dispersive2}
\begin{split}
&\Big|  \int_0^\infty e^{it\lambda} \phi(2^{-k}\lambda)\lambda a_\pm(\lambda d_j)e^{\pm i\lambda d_j} d\lambda\Big|\\&\leq
C_N2^{2k}\big(1+2^k|t\pm d_j|\big)^{-N}(1+2^k d_j)^{-1/2}.
\end{split}
\end{equation}
If $|t|\sim d_j$, we see \eqref{est:k1}.
Otherwise, we have $|t\pm d_j|\geq c|t|$ for some small constant
$c$, choose $N=1$ and $N=0$,  and then use geometric mean argument to prove \eqref{est:k1}.\vspace{0.2cm}

We next prove \eqref{est:k2}. We follow the same lines to obtain 
  \begin{equation}\label{est:k2}
  \begin{split}
\Big|  \int_0^\infty e^{it\lambda} \phi(2^{-k}\lambda)\lambda &\int_0^\infty a_\pm(\lambda d_s)e^{\pm i\lambda d_s}
A_{\sigma}(s,\theta_1,\theta_2) ds \,d\lambda\Big|\\&\leq C 2^{\frac32 k}(2^{-k}+|t|)^{-\frac12} \int_0^\infty |A_{\sigma}(s,\theta_1,\theta_2)| ds.
\end{split}
\end{equation}
From \eqref{Im-A} and \eqref{0-infty}, we have 
\begin{equation}\label{bound-A}
\int_0^\infty |A_{\sigma}(s,\theta_1,\theta_2)| ds\leq C,
\end{equation}
which implies \eqref{est:k2}. Therefore we prove \eqref{est:k}, hence \eqref{dispersive}.

\end{proof}

\begin{center}

\end{center}


\begin{thebibliography}{99}
\addcontentsline{toc}{section}{References}



\bibitem{BFM} {\sc M. D. Blair, G. A. Ford, and J. L. Marzuola}, Strichartz estimates for the wave equation on flat cones, \textit{IMRN}, 2012, 30 pages, doi:10.1093/imrn/rns002.

\bibitem{BPSS}
{\sc N. Burq, F. Planchon, J. Stalker, and A. S.
Tahvildar-Zadeh,} Strichartz estimates for the wave and Schr\"odinger
equations with the inverse-square potential, \textit{J. Funct. Anal.} {\bf 203}
(2003), 519-549.




\bibitem{CK1} {\sc D. Colton and R. Kress}, Inverse acoustic and electromagnetic scattering theory, volume 93 of Applied Mathematical Sciences. Springer-Verlag, Berlin, 1992.


\bibitem{CT1}
 {\sc J. Cheeger, M. Taylor}, Diffraction of waves by Conical Singularities parts I, \textit{Comm.
Pure Appl. Math.} \textbf{35}(1982), 275-331.

\bibitem{CT2}
{\sc J. Cheeger, M. Taylor}, Diffraction of waves by Conical Singularities parts  II, \textit{Comm.
Pure Appl. Math.} \textbf{35}(1982),  487-529.




\bibitem{EGS1}
{\sc M.B. Erdogan, M. Goldberg and W. Schlag},
Strichartz and Smoothing Estimates for Schr\"odinger
Operators with Almost Critical Magnetic Potentials in Three and Higher
Dimensions,
\textit{Forum Math.} {\bf 21} (2009), 687--722.

\bibitem{EGS2}
{\sc M.B. Erdogan, M. Goldberg and W. Schlag},
Strichartz and smoothing estimates for Schr\"odinger operators with
large magnetic potentials in $\R^3$, \textit{J. European Math.
Soc.} {\bf 10} (2008), 507--531.





\bibitem{Ford} {\sc  G. A. Ford}, The fundamental solution and
Strichartz estimates for the Schr\"odinger equation on flat
Euclidean cones, \textit{Comm. Math. Phys.}, 299(2010), 447-467.

\bibitem{FHH} {\sc  G. A. Ford, A. Hassell and L. Hillairet}, Wave propagation on Euclidean surfaces with conical singularities. I: Geometric Diffraction.
\textit{Journal of Spectral Theory}, 8(2018), 605-667.

\bibitem{GH} {\sc C. Guillarmou, and A. Hassell}, Uniform Sobolev estimates for
non-trapping metrics, \textit{J. Inst. Math. Jussieu}, 13(3) (2014), 599--632.

\bibitem{GHS1} {\sc C. Guillarmou, A. Hassell and A. Sikora}, Resolvent at low
energy III: the spectral measure, \textit{Trans. Amer. Math. Soc.},
365(2013), 6103-6148.

\bibitem{GHS2} {\sc C. Guillarmou, A. Hassell and A. Sikora}, Restriction and
spectral multiplier theorems on asymptotically conic manifolds,
\textit{Analysis and PDE}, 6(2013), 893-950.


\bibitem{HL} A. Hassell and P. Lin,  The Riesz transform for homogeneous Schr\"odinger operators on metric cones.
Rev. Mat.Iberoamericana 30(2014),477-522.

\bibitem{HV1} {\sc A. Hassell and A. Vasy}, The spectral projections and the resolvent for scattering metrics, J. d'Analyse Math. 79(1999), 241-298.

\bibitem{HV2} {\sc  A. Hassell and A. Vasy}, The resolvent for Laplace-type operators on asymptotically conic spaces, Ann. Inst. Fourier (Grenoble), 51(2001), 1299-1346.

\bibitem{HZ} {\sc A. Hassell and J. Zhang}, Global-in-time Strichartz estimates on nontrapping asymptotically conic manifolds,  Analysis \& PDE, 9(2016), 151-192.



\bibitem{KT}
{\sc M. Keel and T. Tao}, Endpoint Strichartz estimates,
\textit{Amer. J. Math.}, {\bf 120} (1998), 955-980.

\bibitem{Ko} {\sc H. Kova $\breve{r}\acute{i}$k}, Heat kernels of two-dimensional magnetic Schr\"odinger and Pauli operators, \textit{Calc. Var. Partial Differential Equations}, 
{\bf 44} (2012) 351–374.

\bibitem{KMVZZ}
{\sc R. Killip, C. Miao, M. Visan, J. Zhang, and J. Zheng},
Sobolev spaces adapted to the Schrödinger operator with inverse-square potential, \textit{Math. Z. } {\bf 288} (2018), 1273-1298.

\bibitem{LW}
{\sc A. Laptev, and T. Weidl}, Hardy inequalities for
      magnetic Dirichlet forms, \textit{Mathematical results in quantum
    mechanics (Prague, 1998)}, 299--305; \textit{Oper. Theory Adv. Appl.} {\bf 108},
    Birkh\"auser, Basel, 1999.
    
\bibitem{MOR}    {\sc M. Melgaard, E. Ouhabaz, G. Rozenblum}, Negative discrete spectrum of perturbed multivortex Aharonov-Bohm Hamiltonians, 
\textit{Ann. H. Poincar\'e} {\bf 5} (2004), 979-1012.

\bibitem{MM} {\sc R. Mazzeo and R. B. Melrose}, Pseudodifferential operators on manifolds with fibred boundaries, Asian J. Math.,2(1998), 833-866.

\bibitem{Melrose1} {\sc R. B. Melrose}. Spectral and scattering theory for the Laplacian on asymptotically Euclidian
spaces. Marcel Dekker, 1994.

\bibitem{Melrose2} {\sc R. B. Melrose}. Lecture notes for ?18.157: Introduction to microlocal analysis?. Available at
http://math.mit.edu/ rbm/18.157-F09/18.157-F09.html, 2009.

\bibitem{Melrose3} {\sc R. B. Melrose}. The Atiyah-Patodi-Singer index theorem, volume 4 of Research Notes
in Mathematics. A K Peters Ltd., Wellesley, MA, 1993.

\bibitem{MZ} {\sc R. B. Melrose and M. Zworski}. Scattering metrics and geodesic flow at infinity, Invent. Math. 124(1996), 389-436.

\bibitem{MZZ}
{\sc C. Miao, J. Zhang and J. Zheng}, A note on the cone restriction conjecture,
\textit{Proc. Amer. Math. Soc.}, {\bf 140} (2012),
2091--2102.

\bibitem{Le} {\sc N. N. Lebedev}. Special functions and their applications. Dover Publications Inc., New York, 1972. Revised edition, translated from the Russian and edited by Richard A. Silverman, Unabridged and corrected republication.


\bibitem{RS}
  {\sc M. Reed, and B. Simon,}, \textit{Methods of modern mathematical physics. II. Fourier analysis, self-adjointness}. Academic Press, New York-London, 1975.

%

\bibitem{PT89}
{\sc M. Peshkin, and A. Tonomura}. \textit{The Aharonov-Bohm Effect.} Lect. Notes Phys. {\bf 340} (1989).


\bibitem{PST}
{\sc F. Planchon, J. Stalker and A. S. Tahvildar-Zadeh}, $L^p$ estimates for the wave equation with the inverse-square potential,
\textit{Discrete Contin. Dynam. Systems} {\bf 9} (2003), 427--442.





\bibitem{sogge} {\sc C. D. Sogge}, \textit{Fourier Integrals in Classical Analysis, Cambridge Tracts in Mathematics},
vol. 105, Cambridge University Press, Cambridge, 1993.
%
%
%

\bibitem{Stein} {\sc E. M. Stein}, \textit{Singular Integrals and Differentiability Properties of Functions},
Princeton University Press, Princeton (1970).

\bibitem{Taylor}
{\sc M. Taylor}, \textit{Partial Differential Equations, Vol II},
Springer, 1996.

\bibitem{Watson}  {\sc G. N. Watson}, \textit{A Treatise on the Theory of Bessel Functions. Second Edition}, Cambridge
University Press, 1944.

\bibitem{Zhang}
{\sc J. Zhang}, Strichartz estimates and nonlinear wave equation on nontrapping asymptotically conic manifolds, \textit{Advances in Math.}, 271(2015), 91-111.

\bibitem{ZZ1} {\sc  J. Zhang and J. Zheng}, Global-in-time Strichartz estimates and cubic Schr\"{o}dinger equation in a conical singular space,arXiv:1702.05813


\bibitem{ZZ2}  {\sc  J. Zhang and J. Zheng}, Strichartz estimates and wave equation in a conic singular space,  \textit{Math. Ann.}, 376(2020),525–581.



\end{thebibliography}
\end{document}